\documentclass[11pt]{amsart}
\usepackage{geometry}               
\usepackage{graphicx}
\usepackage{amssymb}
\usepackage{mathabx}
\usepackage{pdfsync}
\usepackage{bbm} 
\usepackage{hyperref}
\usepackage{verbatim} 
\usepackage{epstopdf}
\def\R{\mathbb{R}}

\def\N{\mathbb{N}}

\newcommand{\pvector}[1]{
  \begin{pmatrix}
    #1
  \end{pmatrix}} %
\newcommand{\ddirac}[1]{
  \,{\delta}\!\pvector{#1}\!} %

\renewcommand{\d}{\,{\rm d}} 

\newcommand{\la}{\lambda}
\newcommand{\vphi}{\varphi}

\providecommand{\ab}[1]{\vert #1\vert}
\providecommand{\abs}[1]{\Bigl\vert #1 \Bigr\vert}

\providecommand{\norma}[1]{\Vert #1 \Vert}

\providecommand{\bs}[1]{\boldsymbol{#1}}

\theoremstyle{plain}
\newtheorem{theorem}{Theorem}[section]

\newtheorem{lemma}[theorem]{Lemma}
\newtheorem{proposition}[theorem]{Proposition}

\theoremstyle{definition}

\newtheorem{remark}[theorem]{Remark}

\numberwithin{equation}{section}

\title[A comparison principle for convolution measures]{A comparison principle for convolution measures with applications}

\author[Oliveira e Silva]{Diogo Oliveira e Silva}
\address{
        Diogo Oliveira e Silva:
        School of Mathematics\\
        University of Birmingham\\
        Edgbaston\\
        Birmingham\\
        B15 2TT\\
        England,
        and
Hausdorff Center for Mathematics, Endenicher Allee 62, 53115 Bonn, Germany}
\email{d.oliveiraesilva@bham.ac.uk}
\thanks{\it D.O.S. was  supported by the Hausdorff Center for Mathematics and DFG grant CRC 1060.}

 \author[Quilodr\'an]{Ren\'e Quilodr\'an}
\address{Ren\'e Quilodr\'an}
\email{rquilodr@dim.uchile.cl}

\begin{document}

\subjclass[2010]{42A85, 42B10, 26B10}
\keywords{Sharp Fourier restriction theory, convolution of singular measures.}
\begin{abstract}
We establish the general form of a geometric comparison principle for $n$-fold convolutions of certain singular measures in $\R^d$ which holds for arbitrary $n$ and $d$.
This translates into a pointwise inequality between the convolutions of projection measure on the paraboloid and a perturbation thereof, and we use it to establish a new sharp Fourier extension inequality on a general convex perturbation of a parabola.
Further applications of the comparison principle to sharp Fourier restriction theory are discussed in the companion paper \cite{BOSQ18}.
\end{abstract}

\maketitle

\section{Introduction}

Let $\Sigma$ be a smooth compact hypersurface of $\R^{d+1}$, endowed with a surface-carried measure $\d\mu=\psi\d\sigma$.
Here $\sigma$ denotes the surface measure of $\Sigma$, and the function $\psi$ is smooth and non-negative.
In general, curvature of $\Sigma$ causes the Fourier transform of $\mu$ to decay, which in turn translates into a certain degree of regularity for the convolution powers $\mu^{\ast(n)}$.
To some extent, such considerations apply to the case of non-compact hypersurfaces as well.

In general, the analysis of convolution measures is a hard task.
In the compact setting, the computation reduces to a Fourier inversion, but in practice this is often non-trivial.
If the manifold in question has a large group of symmetries, then computations may become feasible. 
For instance, see
\cite{CFOST17, COS15, CS12a, CS12b, Fo15} for the case of surface measure on spheres and,
in the non-compact setting, 
see \cite{Ca09, Fo07, HZ06} for projection measure on paraboloids,
and 
\cite{COSS17, Fo07, Qu13, Qu15} for the Lorentz invariant measure on cones and hyperboloids.

Understanding convolution measures on {\it perturbations} of these highly symmetric manifolds is of theoretical interest, and naturally arises in applications. 
In \cite{OSQ16}, the authors established a  comparison principle for 2-fold convolutions of certain singular measures.
The purpose of this note is to extend this principle to $n$-fold convolutions,  and to present a sample application in the context of sharp Fourier restriction theory.\\

To state our main result, we introduce some notation.
Given a sufficiently nice function $\phi:\R^d\to\R$, consider the hypersurface in $\R^{d+1}$
\begin{equation}\label{eq:DefSigmaPhi}
\Sigma_\phi=\{({\bf y},|{\bf y}|^2+\phi({\bf y})):{\bf y}\in\R^d\},
\end{equation}
equipped with projection measure
$$\d\nu({\bf y},s)=\ddirac{s-|{\bf y}|^2-\phi({\bf y})}\d {\bf y}\d s.$$
Throughout the  paper, projection measure will be consistently denoted by $\nu$.
We recursively define its $n$-fold convolution  via $\nu^{\ast (2)}=\nu\ast\nu$ and $\nu^{\ast (n)}=\nu\ast\nu^{\ast (n-1)}$.

The following geometric comparison principle is our main result.
It holds in all dimensions $d\geq 1$, and generalizes \cite[Theorem 1.3]{OSQ16} to $n$-fold convolutions, for any $n\geq 2$.

\begin{theorem}\label{thm:Thm7}
For $d\geq 1$, let $\phi:\R^d\to\R$ be a nonnegative, continuously differentiable, strictly convex function. 
Let $\varphi=|\cdot|^2$ and $\psi=|\cdot|^2+\phi$.
Let $\nu_0, \nu$ denote the projection measures on the hypersurfaces $\Sigma_0, \Sigma_\phi$, respectively.
Then, for any integer $n\geq 2$,
\begin{equation}\label{eq:GeometricComparisonMainIneq}
\nu^{\ast(n)}(\boldsymbol{\xi},\tau)\leq 
\nu_0^{\ast(n)}(\boldsymbol{\xi},\tau-n\phi(\boldsymbol{\xi}/n)),
\end{equation}
for every $\boldsymbol{\xi}\in\R^d$ and $\tau>n\psi(\boldsymbol{\xi}/n)$.
Moreover, this inequality is strict at every point in the interior of the support of the measure $\nu^{\ast(n)}$.
\end{theorem}

\noindent Under the assumptions of the theorem, the support of the convolution measure $\nu^{\ast(n)}$ is contained in that of $\nu_0^{\ast(n)}$.
Moreover, both measures define continuous functions inside their supports and, as $\tau\to n\psi(\boldsymbol{\xi}/n)^+$,  the left- and right-hand sides of \eqref{eq:GeometricComparisonMainIneq} approach the boundary values of $\nu^{\ast(n)}$ and $\nu_0^{\ast(n)}$, respectively. 
These assertions follow from Proposition \ref{prop:ConvolutionProperties} below.
We emphasize that, at least when $(d,n)\neq (1,2)$, inequality \eqref{eq:GeometricComparisonMainIneq} is stronger than the mere claim
$$\nu^{\ast(n)}(\boldsymbol{\xi},\tau)\leq 
\nu_0^{\ast(n)}(\boldsymbol{\xi},\tau),\text{ for every } \boldsymbol{\xi}\in\R^d\text{ and }\tau>n\psi(\boldsymbol{\xi}/n).$$
Indeed, the function $\tau\mapsto \nu_0^{\ast(n)}(\boldsymbol{\xi},\tau)$ is non-decreasing,
as observed in Remark \ref{rem:tripleconvolution} below.\\

Connections with Fourier restriction theory \cite{St93, Ta04} and Strichartz estimates for partial differential equations \cite{St77} are to be expected.
The early proof of the Fourier restriction conjecture in the plane due to Fefferman \cite{Fe70} relied on a careful analysis of the convolution measure $f\sigma\ast f\sigma$, where $f$ is a function defined on the circle.
In a different direction, the seminal work of Tomas \cite{To75} used the $TT^*$ method to reduce matters to the study of the operator $f\mapsto f\ast\widehat{\sigma}$.
More recently, a related but distinct comparison principle was used in \cite{RS12} as an effective tool to understand the effects of global smoothing, to derive new estimates for dispersive equations from known ones, and to compare estimates for different equations.

Sharp Fourier restriction theory has received a lot of attention lately, see the recent survey \cite{FOS17} and the references therein.
In \cite{OSQ16}, we used the comparison principle for 2-fold convolutions as the main tool to study a number of questions arising from sharp Fourier restriction theory.
In particular, we computed the optimal constant for the adjoint restriction (or extension) $L^2-L^4$ inequality on the surface $\{({\bf y},|{\bf y}|^2+\phi({\bf y})): {\bf y}\in\R^2\}$,
 and proved that extremizers do not exist.
Our second result is the one-dimensional $L^2-L^6$ analogue of \cite[Theorem 1.2]{OSQ16}.

\begin{theorem}\label{thm:Thm6}
Let $\phi:\R\to\R$ be a nonnegative, twice continuously differentiable, strictly convex function, whose second derivative $\phi''$ satisfies one of the following conditions:
\begin{itemize}
\item[(i)] $\phi''(y_0)=0$, for some $y_0\in\R$, or
\item[(ii)] There exists a sequence ${y_n}\subset\R$ with $|y_n|\to\infty$, as $n\to\infty$, such that $\phi''(y_n)\to0$, as $n\to\infty$.
\end{itemize}
Let $\nu$ denote projection measure on the curve $\Sigma_\phi$.
Then the inequality 
\begin{equation}\label{eq:PertParabolaConvolutionForm}
\|f\nu\ast f\nu\ast f\nu\|^2_{L^2(\R^2)}\leq \frac{{\pi}}{\sqrt{3}} \|f\|_{L^2(\R)}^6
\end{equation}
holds for every $f\in L^2(\R)$, and is sharp.
The sequence $\{f_n\|f_n\|_{L^2}^{-1}\}$ defined via
\begin{equation}\label{eq:ExtSeqPertParabola}
f_n(y):= \left\{ \begin{array}{ll}
e^{-n(\psi(y)-\psi(y_0)-\psi'(y_0)(y-y_0))}, & \textrm{in case \emph{(i)},}\\
e^{-a_n(\psi(y)-\psi(y_n)-\psi'(y_n)(y-y_n))},& \textrm{in case \emph{(ii)},}
\end{array} \right.
\end{equation}
where $\psi:=|\cdot|^2+\phi$ and $\{a_n\}$ is an appropriately chosen sequence, 
is extremizing for  \eqref{eq:PertParabolaConvolutionForm}.
Moreover, extremizers for  \eqref{eq:PertParabolaConvolutionForm} do not exist.
\end{theorem}

\noindent The choice of the sequence $\{a_n\}$  will be clarified in the course of the proof of Theorem \ref{thm:Thm6}, 
which in turn relies on Theorem \ref{thm:Thm7}.

Our methods are further able to resolve a dichotomy from the recent literature concerning the existence of extremizers for certain Strichartz inequalities for higher order Schr\"odinger equations in one spatial dimension.
This question and related ones are explored in the companion paper \cite{BOSQ18}.
\\

{\bf Overview.}
The paper is organized as follows.
In \S \ref{ssec:Convolutions} we establish some useful facts about  convolutions of singular measures.
These are used in \S \ref{ssec:Pointwise} to prove Theorem \ref{thm:Thm7}.
We then prove Theorem \ref{thm:Thm6} in \S \ref{sec:Perturbations}.
\\

{\bf Notation.}
The usual inner product between vectors ${\bf x},{\bf y}\in\R^d$ will be denoted by $\langle {\bf x},{\bf y}\rangle$.
This distinguishes it from the $d\times d$ matrix obtained as the matrix product between  ${\bf x}$ and the transpose of  ${\bf y}$, denoted ${\bf x}\cdot {\bf y}^T$.
The usual matrix product between a $d\times d$ matrix ${\bf A}$ and a vector ${\bf x}\in \R^d$ will likewise be indicated by ${\bf A}\cdot {\bf x}$.

\section{A geometric comparison principle}\label{sec:Comparison}
This section is devoted to the proof of Theorem \ref{thm:Thm7}.
In our analysis, we will make use of the so-called {\it delta-calculus} to perform integration on manifolds, see \cite[Appendix A]{FOS17} for a concise treatment.
\subsection{Convolutions of singular measures}\label{ssec:Convolutions}
The following result is an extension of \cite[Proposition 2.1]{OSQ16} to the case of $n$-fold convolution measures.  
The proof relies on the Implicit Function Theorem.
Other approaches are presumably available,
see \cite{BM97} for a particular instance of the case $n=2$ which instead relies on the co-area formula.

\begin{proposition}\label{prop:ConvolutionProperties}
Let $d\geq 1$ and $n\geq 2$ be integers. 
Let $\psi:\R^d\to\R$ be a strictly convex, nonnegative function of class $C^2(\R^d)$.
Let $\nu$ denote projection measure $\d\nu({\bf y},s)=\ddirac{s-\psi({\bf y})}\d {\bf y}\d s$.
Then the following  holds for the $n$-fold convolution measure $\nu^{\ast(n)}$:
\begin{itemize}
\item[(a)] It is absolutely continuous with respect to Lebesgue measure on $\R^{d+1}$.
\item[(b)] Its support is given by 
\begin{equation}\label{eq:SupportConvolution}
\emph{supp}(\nu^{\ast(n)})=\{(\boldsymbol{\xi},\tau)\in \R^{d+1}: \tau\geq n\psi(\boldsymbol{\xi}/n)\}.
\end{equation}
\item[(c)] Its Radon--Nikodym derivative, also denoted by $\nu^{\ast(n)}$, is given by the formula
\begin{multline}\label{eq:ConvolutionFormula}
\nu^{\ast(n)}(\boldsymbol{\xi},\tau)=\\
\int_{\mathbb{S}^{d(n-1)-1}} \alpha^{d(n-1)-2}\Big(\sum_{i=1}^{n-1}\Big\langle \boldsymbol{\omega}_i, \frac{\nabla\psi(\boldsymbol{\xi}/n+\alpha\sum_{j=1}^{n-1}\boldsymbol{\omega}_j)-\nabla\psi(\boldsymbol{\xi}/n-\alpha\boldsymbol{\omega}_i)}{\alpha}\Big\rangle\Big)^{-1} \d\mu_{\boldsymbol{\omega}},
\end{multline}
provided $\tau>n\psi(\boldsymbol{\xi}/n)$.
Here, 
 $\mu_{\boldsymbol{\omega}}$ denotes surface measure on the unit sphere $\mathbb{S}^{d(n-1)-1}\subset\R^{d(n-1)}$,
$\boldsymbol{\omega}=(\boldsymbol{\omega}_1,\ldots,\boldsymbol{\omega}_{n-1})\in \mathbb{S}^{d(n-1)-1}$, 
$\boldsymbol{\omega}_i\in\R^d$,
and the function $\alpha$ is given by
\begin{equation}\label{eq:AlphaDef}
\alpha(\boldsymbol{\xi},\tau,\boldsymbol{\omega})=
\sqrt{\tau-n\psi(\boldsymbol{\xi}/n)}
\lambda(\sqrt{\tau-n\psi(\boldsymbol{\xi}/n)}\boldsymbol{\omega}),
\end{equation}
where the function $\lambda$ is implicitly defined via identity \eqref{eq:implicitlambda} below.
\item[(d)] It defines a continuous function of the variables $\boldsymbol{\xi},\tau$ in the interior of its support. 
If $(d,n)\neq (1,2)$ and the Hessian matrix of the function $\psi$ satisfies ${\bf H}(\psi)(\boldsymbol{\xi}/n)\neq {\bf 0}$ at some point $\boldsymbol{\xi}\in\R^d$, then the convolution extends continuously to the boundary point $(\boldsymbol{\xi},n\psi(\boldsymbol{\xi}/n))$, with values given by
\begin{equation}\label{eq:Formula3foldConvolution}
(\nu\ast\nu\ast\nu)(\xi,3\psi(\xi/3))=\frac{2\pi}{\sqrt{3}\psi''(\xi/3)},\text{ if } \xi\in\R,
\end{equation}
$$(\nu\ast\nu)(\boldsymbol{\xi},2\psi(\boldsymbol{\xi}/2))=\frac{\pi}{\sqrt{\det({\bf H}(\psi))(\boldsymbol{\xi}/2)}},\text{ if } \boldsymbol{\xi}\in\R^2,$$
$$\nu^{\ast(n)}(\boldsymbol{\xi},n\psi(\boldsymbol{\xi}/n))=0,\text{ if } \boldsymbol{\xi}\in\R^d \text{ and }(d,n)\notin\{(1,2),(1,3),(2,2)\}.$$
If $(d,n)=(1,2)$, then the following asymptotic formula holds:
\begin{equation}\label{eq:AsymptoticFormulad1n2}
(\nu\ast\nu)(\xi,\tau)\asymp \frac{1}{\psi''(\xi/2)\sqrt{\tau-2\psi(\xi/2)}\lambda(\sqrt{\tau-2\psi(\xi/2)})},\text{ as }\tau\downarrow2\psi(\xi/2),
\end{equation}
in the sense that the ratio of the right- and left-hand sides tends to 1, as $\tau\downarrow2\psi(\xi/2)$.
\end{itemize}
\end{proposition}

\begin{proof}
To establish parts (a) and (b), it suffices to consider the case $n=2$, as the general case will then follow by induction. 
This in turn was proved in  \cite[Proposition 2.1]{OSQ16} for $d=2$, but as pointed out there the argument extends to general dimensions. 

We provide the details for parts (c) and (d). 
Let $(\boldsymbol{\xi},\tau)\in\R^{d+1}$ be such that $\tau> n\psi(\boldsymbol{\xi}/n)$. 
Changing variables ${\bf y}_i\rightsquigarrow \boldsymbol{\xi}/n-{\bf y}_i$, $1\leq i< n$, we have
\begin{align}
\nu^{\ast(n)}(\boldsymbol{\xi},\tau)
&=\int_{(\R^d)^n}
\ddirac{\tau-\sum_{i=1}^n\psi({\bf y}_i)}\ddirac{\boldsymbol{\xi}-\sum_{j=1}^n {\bf y}_j}\d 
{\bf y}_1\ldots\d {\bf y}_n\label{eq:FirstIntegralExpressionNuAstN}\\
&=\int_{(\R^d)^{n-1}}\ddirac{\tau-\sum_{i=1}^{n-1}\psi({\bf y}_i)-\psi(\boldsymbol{\xi}-\sum_{j=1}^{n-1}{\bf y}_j)}\d 
{\bf y}_1\ldots\d {\bf y}_{n-1}\notag\\
&=\int_{(\R^d)^{n-1}}\ddirac{\tau-n\psi(\boldsymbol{\xi}/n)-g_n(1,{\bf y})}\d {\bf y}_1\ldots \d {\bf y}_{n-1},\label{eq:IntegralExpressionNuAstN}
\end{align}
where the function $g_n$ is defined on pairs $(t,{\bf y})=(t,({\bf y}_1,\dots,{\bf y}_{n-1}))\in\R\times(\R^d)^{n-1}$ via
$$g_n(t,{\bf y}):=\sum_{i=1}^{n-1}\psi(\boldsymbol{\xi}/n-t{\bf y}_i)+\psi\Bigl(\boldsymbol{\xi}/n+t\sum_{j=1}^{n-1}{\bf y}_j\Bigr)-n\psi(\boldsymbol{\xi}/n).$$ 
We perform another change of variables  ${\bf y}={\bf T}({\bf w})=\lambda {\bf w}$, where $\lambda=\lambda({\bf w})$  is an 
implicit real-valued function of $\mathbf{w}=({\bf w}_1,\ldots,{\bf w}_{n-1})\in (\R^d)^{n-1}$ which takes only strictly positive values if $\mathbf{w}\neq\mathbf{0}$, and is defined via the identity
\begin{equation}\label{eq:implicitlambda}
g_n(1,\lambda {\bf w})=g_n(\lambda,{\bf w})=|{\bf w}|^2.
\end{equation}
For fixed $\boldsymbol{\xi}$, the Intermediate Value Theorem and strict convexity imply that a unique positive 
solution $\lambda=\lambda({\bf w})$ exists if ${\bf w}\neq {\bf 0}$, since $g_n(0,{\bf w})=0$ and $g_n(\lambda,{\bf w})\to\infty$, as 
$\lambda\to\infty$. 
By the Implicit Function Theorem, equation \eqref{eq:implicitlambda} defines $\lambda$ as a $C^1$  function of ${\bf w}$, provided that the derivative of the map
\[\lambda\mapsto \sum_{i=1}^{n-1}
\psi(\boldsymbol{\xi}/n-\lambda {\bf w}_i)
+\psi\Big(\boldsymbol{\xi}/n+\lambda \sum_{j=1}^{n-1}{\bf w}_j\Big)-n\psi(\boldsymbol{\xi}/n)
\]
is nonzero. In view of the strict convexity of the function $\psi$, this is indeed the case if $\lambda>0$. 
See Lemma \ref{lem:transformation-general} below for further details in a slightly more general context.
Since the function $\lambda$ is $C^1$ and ${\bf T}({\bf w})=\lambda({\bf w}){\bf w}$, we have that 
\begin{equation}\label{eq:1Tprime}
{\bf T}'({\bf w})=\lambda {\bf I}+{\bf w}\cdot (\nabla\lambda)^T,
\end{equation}
where ${\bf I}$ stands for the identity matrix in $\R^{d(n-1)}$, the gradient is taken with respect to 
${\bf w}$, and the term ${\bf w}\cdot (\nabla\lambda)^T$ denotes the $d(n-1)\times d(n-1)$ matrix obtained as the 
product 
of the vector ${\bf w}$ 
and the gradient $\nabla\lambda$. 
Implicit differentiation of \eqref{eq:implicitlambda} with respect to ${\bf w}$ yields
\begin{equation}\label{eq:2Tprime}
({\bf T}')^T({\bf w})\cdot {\bf u}=2{\bf w},
\end{equation}
where ${\bf u}=({\bf u}_1,\dotsc,{\bf u}_{n-1})$ is a vector-valued function of $({\bf w},\boldsymbol{\xi})$, given for $1\leq i< n$ by
\begin{equation}\label{eq:3Tprime}
{\bf u}_i=\nabla\psi\Bigl(\boldsymbol{\xi}/n+\lambda\sum_{j=1}^{n-1}{\bf w}_j\Bigr)-\nabla\psi(\boldsymbol{\xi}/n-\lambda {\bf w}_i). 
\end{equation}
From \eqref{eq:1Tprime} and \eqref{eq:2Tprime} it follows that 
\begin{equation}\label{eq:GradientLambda}
\nabla\lambda=\frac{2{\bf w}-\lambda {\bf u}}{\langle {\bf w}, {\bf u}\rangle}.
\end{equation}
Using the Matrix Determinant Lemma, 
\[\det {\bf T}'({\bf w})
=\det (\lambda {\bf I}+\nabla\lambda\cdot {\bf w}^T)
=(1+\lambda^{-1}\langle {\bf w},\nabla \lambda\rangle)\det(\lambda {\bf I}).
\]
Identity \eqref{eq:GradientLambda} then implies
\begin{equation}\label{eq:detTprime}
\det {\bf T}'({\bf w})=\frac{2|{\bf w}|^2}{\langle {\bf w}, {\bf u}({\bf w},\boldsymbol{\xi})\rangle}\lambda({\bf w})^{d(n-1)-1}.
\end{equation}
Note that this is a nonnegative quantity because of the strict convexity of $\psi$. 
Going back to the integral expression \eqref{eq:IntegralExpressionNuAstN}, changing variables as announced, and switching to spherical coordinates, yields
\begin{align*}
\nu^{\ast(n)}(\boldsymbol{\xi},\tau)
&=
\int_{{(\R^d)^{n-1}}}\ddirac{\tau-{n\psi(\boldsymbol{\xi}/n)}-|{\bf w}|^2} \det {\bf T}'({\bf w}) \d {\bf w}\\
&=
\int_{{0}}^\infty\ddirac{\tau-{n\psi(\boldsymbol{\xi}/n)}- r^2} \Big(\int_{\mathbb{S}^{d(n-1)-1}}\det 
{\bf T}'(r\boldsymbol{\omega}) \d\mu_{\boldsymbol{\omega}}\Big) r^{d(n-1)-1} \d r,
\end{align*}
where $\mu_{\boldsymbol{\omega}}$ denotes surface measure on the unit sphere $\mathbb{S}^{d(n-1)-1}$. 
Invoking \eqref{eq:detTprime}, changing variables $r^2\rightsquigarrow s$, and evaluating the inner integral,
\begin{align*}
\nu^{\ast(n)}(\boldsymbol{\xi},\tau)
&=\int_{\mathbb{S}^{d(n-1)-1}}
\Big(\int_{{0}}^\infty 
\ddirac{\tau-{n\psi(\boldsymbol{\xi}/n)}-s}\frac{\sqrt{s}\lambda(\sqrt{s}\boldsymbol{\omega})^{d(n-1)-1}}{\langle\boldsymbol{\omega}, 
{\bf u}(\sqrt{s} \boldsymbol{\omega},\boldsymbol{\xi})\rangle} \sqrt{s}^{{d(n-1)-2}}\d s\Big)
\d\mu_{\boldsymbol{\omega}}\\
&=\int_{\mathbb{S}^{d(n-1)-1}}
\frac{\bigl(\sqrt{\tau-n\psi(\boldsymbol{\xi}/n)}\lambda({\sqrt{\tau-n\psi(\boldsymbol{\xi}/n)}}\boldsymbol{\omega})\bigr)^{d(n-1)-1}}{\langle\boldsymbol{\omega},
{\bf u}({\sqrt{\tau-n\psi(\boldsymbol{\xi}/n)}} \boldsymbol{\omega},\boldsymbol{\xi})\rangle} 
\d\mu_{\boldsymbol{\omega}}.
\end{align*}
Formula \eqref{eq:ConvolutionFormula} now follows from 
the definition \eqref{eq:AlphaDef} of the function $\alpha=\alpha(\boldsymbol{\xi},\tau,\boldsymbol{\omega})$, and 
 the expression \eqref{eq:3Tprime} for the vector ${\bf u}$.
This concludes the verification of part (c).

As for part (d), the continuity of $\nu^{\ast(n)}$ in the interior of its support follows from an inspection of  formula \eqref{eq:ConvolutionFormula}, since the function $\lambda$ is continuous. 
As for boundary values, let us consider the case $(d,n)\neq (1,2)$ first.
Consider a boundary point $(\boldsymbol{\xi}_0,n\psi(\boldsymbol{\xi}_0/n))\in\R^{d+1}$, and suppose that the Hessian matrix ${\bf H}(\psi)(\boldsymbol{\xi}_0/n)$ is nonzero.
We claim that, for fixed $\boldsymbol{\omega}\in\mathbb{S}^{d(n-1)-1}$,
\begin{equation}\label{eq:AlphaLimitBoundary}
\lim_{(\boldsymbol{\xi},\tau)\to(\boldsymbol{\xi}_0,n\psi(\boldsymbol{\xi}_0/n))} \alpha(\boldsymbol{\xi},\tau,\boldsymbol{\omega})=0,
\end{equation}
where the limit is taken over points $(\boldsymbol{\xi},\tau)$ belonging to the interior of the support of $\nu^{\ast(n)}$.
The function $\lambda=\lambda({\bf w})$ satisfies identity \eqref{eq:implicitlambda}, which can be rewritten as
$$g_n(\lambda,{\bf w})=g_n(\alpha,\boldsymbol{\omega})=\tau-n\psi(\boldsymbol{\xi}/n),$$
where $\boldsymbol{\omega}\in \mathbb{S}^{d(n-1)-1}$, ${\bf w}=\sqrt{\tau-n\psi(\boldsymbol{\xi}/n)}\boldsymbol{\omega}$, and 
$\alpha=\alpha(\boldsymbol{\xi},\tau,\boldsymbol{\omega})$ is defined as in \eqref{eq:AlphaDef}. 
As $(\boldsymbol{\xi},\tau)\to(\boldsymbol{\xi}_0,n\psi(\boldsymbol{\xi}_0/n))$ from the interior of the support of $\nu^{\ast(n)}$, the quantity
$\tau-n\psi(\boldsymbol{\xi}/n)$ tends to 0. 
The function $g_n(\cdot,\boldsymbol{\omega})$ attains its unique global 
minimum at $\alpha=0$, where it equals zero. 
It follows that claim \eqref{eq:AlphaLimitBoundary} holds and, as $(\boldsymbol{\xi},\tau)\to (\boldsymbol{\xi}_0,n\psi(\boldsymbol{\xi}_0/n))$, we have
\begin{align}
\sum_{i=1}^{n-1}\Big\langle \boldsymbol{\omega}_i, &\frac{\nabla\psi(\boldsymbol{\xi}/n+\alpha\sum_{j=1}^{n-1}\boldsymbol{\omega}_j)-\nabla\psi(\boldsymbol{\xi}/n-\alpha\boldsymbol{\omega}_i)}{\alpha}\Big\rangle\to\notag\\
&\Big\langle \sum_{i=1}^{n-1} \boldsymbol{\omega}_i, {\bf H}(\psi)(\boldsymbol{\xi}_0/n)\cdot \sum_{i=1}^{n-1}\boldsymbol{\omega}_i\Big\rangle
+\sum_{i=1}^{n-1} \langle \boldsymbol{\omega}_i, {\bf H}(\psi)(\boldsymbol{\xi}_0/n)\cdot \boldsymbol{\omega}_i\rangle.\label{eq:LimitWithHessian}
\end{align}
This is a strictly positive quantity since $\psi$ is strictly convex and ${\bf H}(\psi)(\boldsymbol{\xi}_0/n)\neq 0$, and from formula \eqref{eq:ConvolutionFormula} we see that $\nu^{\ast(n)}(\boldsymbol{\xi},n\psi(\boldsymbol{\xi}/n))$ vanishes identically, except possibly when $d(n-1)=2$, i.e. $(d,n)\in\{(2,2),(1,3)\}$.
The former case was treated in \cite[Proposition 2.1]{OSQ16}, so we focus on the latter.
For every $\xi\in\R$, we have
\begin{align*}
(\nu\ast\nu\ast\nu)(\xi,3\psi(\xi/3))
&=\int_{\mathbb{S}^{1}} \frac{1}{\psi''(\xi/3)(\omega_1+\omega_2)^2+\psi''(\xi/3)(\omega_1^2+\omega_2^2)}
\d\mu_{(\omega_1,\omega_2)}\\
&=\frac{1}{\psi''(\xi/3)}\int_{0}^{2\pi}\frac{1}{2+2\sin\theta\cos\theta}\d\theta
=\frac{2\pi}{\sqrt{3}\,\psi''(\xi/3)},
\end{align*}
as claimed.
Finally, if $(d,n)=(1,2)$, then expression \eqref{eq:LimitWithHessian} equals $2\psi''(\xi_0/2)$.
Noting that the function $\lambda$ is even if $n=2$, we obtain \eqref{eq:AsymptoticFormulad1n2}.
The proof is now complete.
\end{proof}

\begin{remark}\label{rem:tripleconvolution}
Let us specialize to the case of the unperturbed paraboloid $\psi=|\cdot|^2$.
It was observed in \cite[Remark 2.2]{OSQ16} that formula \eqref{eq:ConvolutionFormula} for $d=n=2$ recovers the result from \cite[Lemma 3.2]{Fo07} for the 2-fold convolution of projection measure on the two-dimensional paraboloid.
In a similar way, if $(d,n)=(1,3)$, then the expression for the 3-fold convolution of projection measure $\nu_0$ on the parabola $\{\tau=\xi^2\}\subset\R^2$ given by formula \eqref{eq:ConvolutionFormula} reduces to 
$$(\nu_0\ast\nu_0\ast\nu_0)(\xi,\tau)
=\frac12\int_{\mathbb{S}^1}\frac{1}{\omega_1(2\omega_1+\omega_2)+\omega_2(\omega_1+2\omega_2)}\d\mu_{(\omega_1,\omega_2)}=\frac{\pi}{\sqrt{3}},$$
provided  $\tau>{\xi^2}/3$.
This recovers the value obtained in \cite[Lemma 4.1]{Fo07}.
For general $(d,n)$, one can check that, in the case $\psi=|\cdot|^2$, the function $\la=\la(\mathbf{w})$ implicitly defined by identity \eqref{eq:implicitlambda} is given by
\[ \la(\mathbf{w})=\frac{\ab{\mathbf{w}}}{\Bigl(\abs{\sum_{i=1}^{n-1}{\bf w}_i}^2+\sum_{j=1}^{n-1}\ab{{\bf w}_j}^2\Bigr)^{\frac12}}.\]
This is a homogeneous function of degree zero, and so the function $\alpha=\alpha(\bs\xi,\tau,\bs\omega)$ defined in \eqref{eq:AlphaDef} is given by
 $\alpha(\bs\xi,\tau,\bs\omega)=(\tau-\ab{\bs\xi}^2/n)^{\frac12}\,\la(\bs\omega)$.
Consequently, if $\tau>|\boldsymbol{\xi}|^2/n$, then
\begin{align*}
\nu_0^{\ast(n)}(\boldsymbol{\xi},\tau)&=
\frac{1}{2}\int_{\mathbb{S}^{d(n-1)-1}} \alpha^{d(n-1)-2}\Big(\sum_{i=1}^{n-1}\Big\langle \boldsymbol{\omega}_i, \sum_{j=1}^{n-1}\boldsymbol{\omega}_j+\boldsymbol{\omega}_i\Big\rangle\Big)^{-1} \d\mu_{\boldsymbol{\omega}}\\
&=\frac{1}{2}(\tau-\ab{\boldsymbol{\xi}}^2/n)^{\frac{d(n-1)}{2}-1}\int_{\mathbb{S}^{d(n-1)-1}} \la(\bs\omega)^{d(n-1)-2}\Big(\sum_{i=1}^{n-1}\Big\langle \boldsymbol{\omega}_i, \sum_{j=1}^{n-1}\boldsymbol{\omega}_j+\boldsymbol{\omega}_i\Big\rangle\Big)^{-1} \d\mu_{\boldsymbol{\omega}}\\
&=\frac{1}{2}(\tau-\ab{\boldsymbol{\xi}}^2/n)^{\frac{d(n-1)}{2}-1}\int_{\mathbb{S}^{d(n-1)-1}} \Big(\abs{\sum_{i=1}^{n-1} \boldsymbol{\omega}_i}^2 +\sum_{j=1}^{n-1}\ab{\boldsymbol{\omega}_j}^2\Big)^{-\frac{d(n-1)}{2}} \d\mu_{\boldsymbol{\omega}}.
\end{align*}
The latter integral can be computed in polar coordinates, see \cite{BBI15, Ca09}.
 Alternatively, the value of the constant $c_{d,n}$ in the expression
 \begin{equation}\label{eq:PreNFoldProj}
\nu_0^{\ast(n)}(\boldsymbol{\xi},\tau)=c_{d,n}(\tau-\ab{\boldsymbol{\xi}}^2/n)_+^{\frac{d(n-1)}{2}-1}
\end{equation}
can be determined by simply multiplying both sides of \eqref{eq:PreNFoldProj} by the factor $\exp(-\tau)$ and integrating in $\boldsymbol{\xi},\tau$.
Indeed, recall \eqref{eq:FirstIntegralExpressionNuAstN} and observe that
 \begin{align*}
 \int_{\R^{d+1}} e^{-\tau}& \nu_0^{\ast(n)}(\boldsymbol{\xi},\tau) \d\boldsymbol{\xi}\d\tau\\
& =\int_{\R^{d+1}} \int_{(\R^d)^n} e^{-\tau}\ddirac{\tau-\sum_{i=1}^n|{\bf y}_i|^2}\ddirac{\boldsymbol{\xi}-\sum_{j=1}^n {\bf y}_j}\d {\bf y}_1\ldots\d {\bf y}_n \d\boldsymbol{\xi}\d\tau\\
& =\int_{(\R^d)^n} e^{-\sum_{i=1}^{n}|{\bf y}_i|^2}\d {\bf y}_1\ldots\d {\bf y}_n=\pi^{\frac{dn}2}.
 \end{align*} 
 On the other hand, a simple change of variables yields
 $$ \int_{\R^{d+1}} e^{-\tau} (\tau-\ab{\boldsymbol{\xi}}^2/n)_+^{\frac{d(n-1)}{2}-1} \d\boldsymbol{\xi}\d\tau=(n\pi)^{\frac d2}\Gamma(\tfrac{d(n-1)}2),$$
and therefore
$$c_{d,n}=\frac{\pi^{\frac{d(n-1)}2}}{n^{\frac d2}\Gamma(\frac{d(n-1)}2)}.$$
In particular, Proposition \ref{prop:ConvolutionProperties} generalizes the formula obtained in \cite[Lemma 2.4]{BBI15}.
Moreover, we see from \eqref{eq:PreNFoldProj} that $\nu_0^{\ast(n)}(\boldsymbol{\xi},\cdot)$ defines a non-decreasing function of $\tau$ on the region $\{\tau>|\boldsymbol{\xi}|^2/n\}$, for every fixed $\boldsymbol{\xi}\in\R^d$ and $(d,n)\neq (1,2)$.
\end{remark}

\begin{remark}\label{rem:WeightedVersion}
From the proof of Proposition \ref{prop:ConvolutionProperties}, it is clear that a similar statement holds in the weighted setting. 
Let $w:\R^d\to\R$ be a continuous function. 
Parts (a)--(d) in the statement of Proposition \ref{prop:ConvolutionProperties} hold for the convolution measure $(w\nu)^{\ast(n)}$, with minor modifications which we now indicate.
Firstly, in general only inclusion $\subseteq$ holds in \eqref{eq:SupportConvolution} instead of equality; there is equality if $w>0$. 
Secondly, defining
\[ W(\boldsymbol{\xi};{\bf y}_1,\dotsc,{\bf y}_{n-1})
:=w(\boldsymbol{\xi}/n+{\bf y}_1+\dotsb+{\bf y}_{n-1})\prod_{i=1}^{n-1}w(\boldsymbol{\xi}/n- {\bf y}_i), \]
we have the substitute formula for \eqref{eq:ConvolutionFormula},
\begin{align*} 
(w\nu)^{\ast(n)}(\boldsymbol{\xi},\tau)
&=\int_{\mathbb{S}^{d(n-1)-1}}\alpha^{d(n-1)-2}W(\boldsymbol{\xi};\alpha\boldsymbol{\omega}_1,\dots,\alpha\boldsymbol{\omega}_{n-1})\\
&\qquad\biggl(\sum_{i=1}^{n-1}\Big\langle \omega_i, \frac{\nabla 
	\psi(\boldsymbol{\xi}/n+\alpha\sum_{j=1}^{n-1}\boldsymbol{\omega}_j)-\nabla 
	\psi(\boldsymbol{\xi}/n-\alpha\boldsymbol{\omega}_i)}{\alpha}\Big\rangle\biggr)^{-1}
\d\mu_{\boldsymbol{\omega}},
\end{align*}
where $\alpha$ is defined in \eqref{eq:AlphaDef} and is independent of the weight $w$. 
Lastly, the convolution $(w\nu)^{\ast(n)}$ defines a continuous function of 
$\boldsymbol{\xi},\tau$ in the interior of its support. If $(d,n)\neq (1,2)$  and the matrix ${\bf H}(\psi)(\boldsymbol{\xi}/n)$ is nonzero, 
then the convolution  
extends continuously to the boundary point $(\boldsymbol{\xi},n\psi(\boldsymbol{\xi}/n))$,
with values given by
$$(w\nu\ast w\nu\ast w\nu)(\xi,3\psi(\xi/3))=\frac{2\pi w(\xi/3)^3}{\sqrt{3}\,\psi''(\xi/3)},\text{ if 
}\xi\in\R,$$
$$(w\nu\ast w\nu)(\boldsymbol{\xi},2\psi(\boldsymbol{\xi}/2))
={\frac{\pi w(\boldsymbol{\xi}/n)^{2}}{\sqrt{\det({\bf H}(\psi)(\boldsymbol{\xi}/2))}}},\text{ if }\boldsymbol{\xi}\in\R^{2},$$
$$(w\nu)^{\ast(n)}(\boldsymbol{\xi},n\psi(\boldsymbol{\xi}/n))= 0, \text{ if }\boldsymbol{\xi}\in\R^{d}  \text{ and }(d,n)\notin\{(1,2),(1,3),(2,2)\}.$$
If $(d,n)=(1,2)$, then the following asymptotic formula holds:
\begin{equation*}\label{weight-asymp-d1n2}
(w\nu\ast w\nu) (\xi,\tau)\asymp 
\frac{w(\xi/2)^2}{\psi''(\xi/2)\sqrt{\tau-2\psi(\xi/2)}\;\lambda(\sqrt{\tau-2\psi(\xi/2)})},
\text{ as } \tau\downarrow 2\psi(\xi/2),
\end{equation*}
again in the sense that the ratio of the right- and left-hand sides tends to $1$, as $\tau\downarrow 2\psi(\xi/2)$.
\end{remark}

\subsection{A pointwise inequality for convolution measures}\label{ssec:Pointwise}
It was remarked in \cite[\S 3]{OSQ16} that $n$-linear versions of  \cite[Lemmata 3.2 and 3.3]{OSQ16} seemed more intricate if $n\geq 3$.
Here we overcome this difficulty by constructing appropriate inverses to the maps $\lambda$ and ${\bf T}$ considered in the previous subsection.
This leads to a proof of  Theorem \ref{thm:Thm7}.

Let $d\geq 1$ and $n\geq 2$.
Consider two convex functions $\psi,\varphi:\R^d\to\R$.
 Given $\boldsymbol{\xi}\in\R^d$ and ${\bf y}=({\bf y}_1,\dotsc,{\bf y}_{n-1})$ with ${\bf y}_i\in\R^d$, $1\leq i<n$, define the following auxiliary 
 functions acting on pairs $(t,{\bf y})\in\R\times\R^{d(n-1)}$:
 \begin{align}
 g_n(t,{\bf y})&:=\sum_{i=1}^{n-1}\psi(\boldsymbol{\xi}/n-t{\bf y}_i)+\psi\Bigl(\boldsymbol{\xi}/n+t\sum_{j=1}^{n-1}{\bf y}_j\Bigr)-n\psi(\boldsymbol{\xi}/n), \label{eq:DefGn}\\
 h_n(t,{\bf y})&:=\sum_{i=1}^{n-1}\varphi(\boldsymbol{\xi}/n-t{\bf y}_i)+\varphi\Bigl(\boldsymbol{\xi}/n+t\sum_{j=1}^{n-1}{\bf y}_j\Bigr)-n\varphi(\boldsymbol{\xi}/n). \label{eq:DefHn}
 \end{align}
 We are omitting the dependence on $\boldsymbol{\xi}$, which we assume to be fixed.
 By another slight abuse of notation, we will sometimes drop the dependence on ${\bf y}$ and simply write $g_n(t)=g_n(t,{\bf y})$, and similarly for $h_n$.
 Note that $g_n=h_n\equiv 0$ if ${\bf y}={\bf 0}$.
 The following generalization of \cite[Lemma 3.1]{OSQ16} holds.
 
  \begin{lemma}\label{lem:n-convexity}
 	Let $d\geq 1$ and $n\geq 2$. Let $\psi,\varphi:\R^d\to\R$ be differentiable, convex functions, 
 	such that their difference $\psi-\varphi$ is also  convex.
 	Given $\boldsymbol{\xi},{\bf y}_1,\ldots,{\bf y}_{n-1}\in\R^d$, 
	define the functions $g_n, h_n$ as above. 
	Write ${\bf y}=({\bf y}_1,\ldots,{\bf y}_{n-1})$.
	Then:
 	\begin{enumerate}
 		\item[(a)]  $g_n(t)\geq h_n(t)\geq 0$, for every $t\in\R$.
 		\item[(b)] The functions $g_n$ and $h_n$ are convex.
 		\item[(c)] $g_n'(0)=h_n'(0)=0$.
 		\item[(d)] If $\psi$ is strictly convex and ${\bf y}\neq {\bf 0}$, then 
 		$g_n$ attains its unique global minimum at $t=0$.
 		\item[(e)] If $\psi$ is strictly convex and ${\bf y}\neq {\bf 0}$, then 
 		there exists a unique nonnegative $\lambda=\lambda({\bf y},\boldsymbol{\xi})$ such that
 		\begin{equation}\label{eq:DefLambdaViaHG}
		h_n(1,{\bf y})=g_n(\lambda,{\bf y}),
		\end{equation}
 		and moreover $0\leq \lambda\leq 1$.
 		\item[(f)] If $\varphi$ is strictly convex and ${\bf y}\neq {\bf 0}$, 
 		then there exists a unique nonnegative $\rho=\rho({\bf y},\boldsymbol{\xi})$ such that
 		\begin{equation}\label{eq:DefRhoViaHG}
		h_n(\rho,{\bf y})=g_n(1,{\bf y}),
		\end{equation}
		and moreover $\rho\geq 1$.
 		\item[(g)] If $h_n(1,{\bf y})>0$, then $\lambda({\bf y},\boldsymbol{\xi})>0$. If $h_n(1,{\bf y})<g_n(1,{\bf y})$, then 
 		$\lambda({\bf y},\boldsymbol{\xi})<1$ and 
 		$\rho({\bf y},\boldsymbol{\xi})>1$.
 		\item[(h)] Assume that $\psi, \varphi$ are strictly convex functions, and let ${\bf y}\neq {\bf 0}$. 
 		Letting ${\bf u}=\lambda({\bf y},\boldsymbol{\xi}){\bf y}$ and ${\bf v}=\rho({\bf y},\boldsymbol{\xi}){\bf y}$, 
		we have
 		\[ \rho({\bf u},\boldsymbol{\xi})\lambda({\bf y},\boldsymbol{\xi})=1=\rho({\bf y},\boldsymbol{\xi})\lambda({\bf v},\boldsymbol{\xi}). \]
  	\end{enumerate}
 \end{lemma}

 \begin{proof}
  	We focus on part (h) as the other assertions follow easily from an adaptation of the proof of \cite[Lemma 3.1]{OSQ16}.
	Since 
 	${\bf y}=({\bf y}_1,\dotsc,{\bf y}_{n-1})\neq ({\bf 0},\dotsc,{\bf 0})$, the strict convexity of $\varphi$ implies $h_n(1,{\bf y})>0$, 
 	and so from part (g) it follows that $\lambda({\bf y},\boldsymbol{\xi})>0$. 
	In particular,  the vector 
 	${\bf u}=({\bf u}_1,\dotsc,{\bf u}_{n-1})=\lambda({\bf y}_1,\dots,{\bf y}_{n-1},\boldsymbol{\xi})({\bf y}_1,\dots,{\bf y}_{n-1})\in (\R^d)^{n-1}$ is nonzero. 
	By definition, the function $\rho({\bf u},\boldsymbol{\xi})$ is the unique solution 
 	of the equation 
 	$h_n(\rho,{\bf u})=g_n(1,{\bf u})$.
 	\noindent Note that $g_n(1,{\bf u})=g_n(1,\lambda {\bf y})=g_n(\lambda,{\bf y})$. 
	By definition of $\lambda({\bf y},\boldsymbol{\xi})$ in part (e), we have 
 	$g_n(\lambda,{\bf y})=h_n(1,{\bf y})$, and one easily checks that 
	$h_n(1,{\bf y})=h_n(1/\lambda,\lambda {\bf y})=h_n(1/\lambda,{\bf u})$. Thus
 	\[ h_n(\tfrac1\lambda,{\bf u})=g_n(1,{\bf u}). \]
 	By definition of $\rho({\bf u},\boldsymbol{\xi})$ and the  uniqueness statement in part (f), we 
 	necessarily have $\rho({\bf u},\boldsymbol{\xi})=1/\lambda({\bf y},\boldsymbol{\xi})$. 
	The second part follows in a similar way, noting that
 	\[ h_n(1,{\bf v})=h_n(1,\rho {\bf y})=h_n(\rho,{\bf y})=g_n(1,{\bf y})=g_n(1/\rho, \rho {\bf y})=g_n(1/\rho,{\bf v}), \]
 	so that $\lambda({\bf v},\boldsymbol{\xi})=1/\rho({\bf y},\boldsymbol{\xi})$.
 \end{proof}
 
 Henceforth we restrict attention to strictly convex,  $C^1$ functions $\psi, \varphi:\R^d\to\R$, and introduce two sets which play a  role in the proof of Theorem \ref{thm:Thm7}. Given $\boldsymbol{\xi}\in\R^d$ and 
$c\in\R$, define the {\it ellipsoids} 
\begin{align*}\label{eq:psi-ellipsoid}
 \mathcal{E}_\psi(\boldsymbol{\xi},c)&:=\{{\bf y}\in\R^{d(n-1)}:g_n(1,{\bf y})=c\}, \\
 \mathcal{E}_\varphi(\boldsymbol{\xi},c)&:=\{{\bf y}\in\R^{d(n-1)}:h_n(1,{\bf y})=c\}.
 \end{align*}
The sets  $\mathcal E_\psi(\boldsymbol{\xi},c)$ and $\mathcal E_\varphi(\boldsymbol{\xi},c)$ are non-empty provided $c\geq 0$, and codimension 1 hypersurfaces if $c>0$. 
Moreover, for each fixed $\boldsymbol{\xi}\in\R^d$, the disjoint union of the ellipsoids 
$\mathcal E_\psi(\boldsymbol{\xi},c)$ as the parameter $c\geq 0$ ranges over the nonnegative real numbers equals 
the whole of $(\R^d)^{n-1}$, and similarly for $\varphi$.
Now, for each $\boldsymbol{\xi}\in\R^d$, define the map
\begin{equation}\label{eq:n-T}
{\bf T}\colon\R^{d(n-1)}\setminus\{0\}\to\R^{d(n-1)}\setminus\{0\},\quad {\bf T}({\bf y}):=\lambda({\bf y},\boldsymbol{\xi}){\bf y},
\end{equation} 
where $\lambda$ is the unique positive solution of \eqref{eq:DefLambdaViaHG}.
We also define the map 
\begin{equation}\label{eq:n-S}
{\bf S}\colon\R^{d(n-1)}\setminus\{0\}\to\R^{d(n-1)}\setminus\{0\},\quad {\bf S}({\bf y}):=\rho({\bf y},\boldsymbol{\xi}){\bf y},
\end{equation}
where $\rho$ is the unique positive 
solution of \eqref{eq:DefRhoViaHG}.
That ${\bf T}, {\bf S}$ are well-defined follows from the strict convexity of $\psi, \varphi$, together with Lemma 
\ref{lem:n-convexity}. 
Further properties of the maps ${\bf T, S}$ are contained in the following lemma.

\begin{lemma}\label{lem:transformation-general}
	Let $\psi,\varphi:\R^d\to\R$ be strictly convex, $C^1$ functions 
	with a convex difference $\psi-\varphi$.
	Let $\boldsymbol{\xi}\in \R^d$ be given, and consider the transformations ${\bf T}$ and ${\bf S}$ given by 
	\eqref{eq:n-T} and \eqref{eq:n-S}, respectively. Then:
	\begin{enumerate}
		\item[(a)] ${\bf T}$ and ${\bf S}$ are inverse maps. 
		\item[(b)] ${\bf T}$ and ${\bf S}$ are continuously differentiable.
		\item[(c)] If ${\bf T}'({\bf y})$ and ${\bf S}'({\bf y})$ respectively denote the Jacobian matrices of ${\bf T}$ and ${\bf S}$ at a point ${\bf y}\neq {\bf 0}$, then 
		\begin{align*}
		\det {\bf T}'({\bf y})
		&=\la({\bf y})^{d(n-1)-1}\frac{\sum_{i=1}^{n-1}\bigl\langle\nabla\vphi(\boldsymbol{\xi}/n+\sum_{j=1}^{n-1}{\bf y}_j)-
		 \nabla\vphi (\boldsymbol{\xi}/n-{\bf y}_i), 
			{\bf y}_i\bigr\rangle}{\sum_{i=1}^{n-1}\bigl\langle\nabla\psi(\boldsymbol{\xi}/n+\la\sum_{j=1}^{n-1}{\bf y}_j)- 
			\nabla\psi 
			(\boldsymbol{\xi}/n-\la {\bf y}_i), {\bf y}_i\bigr\rangle},\\
		\det {\bf S}'({\bf y})
		&=\rho({\bf y})^{d(n-1)-1}\frac{\sum_{i=1}^{n-1}\bigl\langle\nabla\psi(\boldsymbol{\xi}/n+\sum_{j=1}^{n-1}{\bf y}_j)-
			\nabla\psi (\boldsymbol{\xi}/n-{\bf y}_i), 
			{\bf y}_i\bigr\rangle}{\sum_{i=1}^{n-1}\bigl\langle\nabla\vphi(\boldsymbol{\xi}/n+\rho\sum_{j=1}^{n-1}{\bf y}_j)-
			 \nabla\vphi (\boldsymbol{\xi}/n-\rho {\bf y}_i), {\bf y}_i\bigr\rangle}.
		\end{align*} 
		\item[(d)]  If $c>0$, then ${\bf T}$ defines a bijection from $\mathcal E_\varphi(\boldsymbol{\xi},c)$ onto $\mathcal E_\psi(\boldsymbol{\xi},c)$, and ${\bf S}$ defines a bijection from $\mathcal E_\psi(\boldsymbol{\xi},c)$ onto  $\mathcal E_\varphi(\boldsymbol{\xi},c)$	.
		\end{enumerate}
\end{lemma}

\begin{proof}
Part (a) is a restatement of part (h) in Lemma \ref{lem:n-convexity}. Indeed, if ${\bf y}\neq{\bf 0}$, then
	$\rho(\la({\bf y}){\bf y})\la({\bf y})=1$, and this implies
		\[{\bf S}({\bf T}({\bf y}))=\rho({\bf T}({\bf y})){\bf T}({\bf y})=\rho(\la({\bf y}){\bf y})\la({\bf y}){\bf y}={\bf y}.\]
	A similar argument  shows that ${\bf T}({\bf S}({\bf y}))={\bf y}$, for every nonzero vector ${\bf y}\in\R^{d(n-1)}$.
	
	 Part (b) follows from the Implicit Function Theorem, after verifying that the 
	derivative of the map 
	$t\mapsto g_n(t,{\bf y})$
	is nonzero for each ${\bf y}\in\R^{d(n-1)}\setminus\{0\}$, provided $t>0$. 
	This derivative equals 
	\[g_n'(t,{\bf y})=\sum_{i=1}^{n-1}\Bigl\langle\nabla\psi\Big(\boldsymbol{\xi}/n+t\sum_{j=1}^{n-1}{\bf y}_j\Big)- \nabla\psi 
	(\boldsymbol{\xi}/n-t{\bf y}_i), {\bf y}_i\Bigr\rangle,\]
	which is nonzero if ${\bf y}\neq {\bf 0}$ because of the strict convexity of $\psi$. 
	
	 To verify (c), we compute the Jacobian matrix of the map ${\bf T}$ analogously to what was  
	done in the proof of Proposition \ref{prop:ConvolutionProperties}. 
	Implicit differentiation of  identity \eqref{eq:DefLambdaViaHG} with respect 
	to ${\bf y}$ yields 
	\[(\la {\bf I}+\nabla \la\cdot  {\bf y}^T)\cdot {\bf u}={\bf v},\]
	where the components of the vectors ${\bf u}=({\bf u}_1,\dotsc,{\bf u}_{n-1}),\, 
	{\bf v}=({\bf v}_1,\dotsc,{\bf v}_{n-1})$ equal
	\begin{align*}
	{\bf u}_i=\nabla\psi\Big(\boldsymbol{\xi}/n+\la\sum_{j=1}^{n-1}{\bf y}_j\Big)-\nabla\psi(\boldsymbol{\xi}/n-\la {\bf y}_i),
	\text{ and }
	{\bf v}_i=\nabla\vphi\Big(\boldsymbol{\xi}/n+\sum_{j=1}^{n-1}{\bf y}_j\Big)-\nabla\vphi(\boldsymbol{\xi}/n-{\bf y}_i),
	\end{align*}
	for $1\leq i<n$.
	It follows that
	$\nabla \la=\frac{{\bf v}-\la {\bf u}}{\langle {\bf u},{\bf y}\rangle},$
		with strictly positive denominator $\langle {\bf u}, {\bf y}\rangle$ if ${\bf y}\neq {\bf 0}$.
	In a similar way, we find
	$\nabla\rho=\frac{{\bf b}-\rho {\bf a}}{\langle {\bf a},{\bf y}\rangle}$,
	where the components of the vectors ${\bf a}=( {\bf a}_1,\dotsc,{\bf a}_{n-1}),\,{\bf b}=({\bf b}_1,\dotsc,{\bf b}_{n-1})$ equal
		\begin{align*}
	{\bf a}_i=\nabla\vphi\Big(\boldsymbol{\xi}/n+\rho\sum_{j=1}^{n-1}{\bf y}_j\Big)-\nabla\vphi(\boldsymbol{\xi}/n-\rho {\bf y}_i),
	\text{ and }
	{\bf b}_i=\nabla\psi\Big(\boldsymbol{\xi}/n+\sum_{j=1}^{n-1}{\bf y}_j\Big)-\nabla\psi(\boldsymbol{\xi}/n-{\bf y}_i).
	\end{align*}
	Using the above expressions for $\nabla\la$ and $\nabla\rho$ together with the Matrix Determinant 
	Lemma, 
	\begin{align*}
	\det {\bf T}'({\bf y})&=\det(\la {\bf I}+\nabla\la\cdot {\bf y}^T)
	=\det(\lambda {\bf I})(1+\la^{-1} \langle {\bf y}, 
	\nabla\la\rangle)\\
	&=\la^{d(n-1)}(1+\la^{-1}\langle\nabla\la,{\bf y}\rangle)
	=\la^{d(n-1)-1}\frac{\langle {\bf v},{\bf y}\rangle}{\langle {\bf u},{\bf y}\rangle},
	\end{align*}
	and similarly
	\[\det {\bf S}'({\bf y})
	=\rho^{d(n-1)}(1+\rho^{-1}\langle\nabla\rho,{\bf y}\rangle)
	=\rho^{d(n-1)-1}\frac{\langle {\bf b},{\bf y}\rangle}{\langle {\bf a},{\bf y}\rangle}.\]
	
	 We finally turn to part (d). That the maps ${\bf T}$ and ${\bf S}$ have the desired mapping 
	properties  follows from the defining identities
	\eqref{eq:DefLambdaViaHG} and \eqref{eq:DefRhoViaHG}, respectively.
	In view of  (a), the restriction of ${\bf T}$ (resp. ${\bf S}$) to the set $\mathcal E_\varphi$ (resp. $\mathcal E_\psi$) is a bijective map. 
\end{proof}

We can rewrite the Jacobian determinants of ${\bf T}$ and ${\bf S}$ as
\begin{equation}\label{eq:equivalent-T'S'}
\det {\bf T}'({\bf y})=\la({\bf y})^{d(n-1)-1}\frac{h_n'(1)}{g_n'(\la({\bf y}))},\quad 
\det {\bf S}'({\bf y})=\rho({\bf y})^{d(n-1)-1}\frac{g_n'(1)}{h_n'(\rho({\bf y}))}. 
\end{equation}
Here we are using the facts that
$h_n'(1)=\langle 
{\bf v},{\bf y}\rangle$, $g_n'(\la)=\langle {\bf u},{\bf y}\rangle$, $h_n'(\rho)=\langle {\bf a},{\bf y}\rangle$ and $g_n'(1)=\langle 
{\bf b},{\bf y}\rangle$, where ${\bf u},{\bf v},{\bf a},{\bf b}$ are the vectors introduced in the course of the proof of Lemma \ref{lem:transformation-general}.
Specializing to the case $\vphi=|\cdot|^2$, the expression defining $h_n$ 
simplifies to
\begin{align*}
h_n(t)
=t^2\Bigl(\abs{\sum_{j=1}^{n-1}{\bf y}_j}^2+\sum_{j=1}^{n-1}\ab{{\bf y}_j}^2\Bigr)
\end{align*}
and so  $h_n(t)=t^2 h_n(1)$ and $h_n'(t)=2th_n(1)$. 
In particular, $\la h_n'(1)=h_n'(\la)$, and \eqref{eq:equivalent-T'S'} becomes
\begin{equation}\label{quadratic-T'S'}
\det {\bf T}'({\bf y})=\la({\bf y})^{d(n-1)-2}\frac{h_n'(\la({\bf y}))}{g_n'(\la({\bf y}))},\quad \det 
{\bf S}'({\bf y})=\rho({\bf y})^{d(n-1)-2}\frac{g_n'(1)}{h_n'(1)}.
\end{equation}
Note also  that $\rho({\bf y})$, which solves 
$h_n(\rho)=g_n(1)$, is given by
\begin{equation}\label{quadratic-rho}
 \rho({\bf y})=\biggl(\frac{g_n(1)}{h_n(1)}\biggr)^{\frac12}.
\end{equation}
We are now ready for our next lemma.

 \begin{lemma}\label{lem:n-contractive}
 	Let $d\geq 1$ and $n\geq 2$.
 	Let $\vphi=|\cdot|^2$ and $\psi=|\cdot|^2+\phi$, where $\phi\geq0$ is a strictly convex  
 	$C^1(\R^d)$ function.
 	Let $\boldsymbol{\xi}\in \R^d$ be given, and consider the map ${\bf T}$  given by 
 	\eqref{eq:n-T}.
 	Then 
 	\begin{equation}\label{n-<1}
 	|\det {\bf T}'({\bf y})|< 1, \textrm{ for every }{\bf y}\neq {\bf 0}.
 	\end{equation}
 \end{lemma}
 
  \begin{proof}
 	Fix ${\bf y}\neq {\bf 0}$.
 	For the particular choices of $\psi,\vphi$ as in the statement of the lemma, define functions $g_n$ and $h_n$ via  \eqref{eq:DefGn} and \eqref{eq:DefHn}, respectively.
 	Recall the first identity in \eqref{quadratic-T'S'}:
 	\begin{equation}\label{eq:n-hmgDetT}
 	\det 
 	{\bf T}'({\bf y})=\lambda({\bf y})^{d(n-1)-2}\frac{h_n'(\lambda({\bf y}))}{g_n'(\lambda({\bf y}))}.
 	\end{equation}
 	We have already argued that $g_n-h_n$ is a nonnegative,  differentiable, strictly convex function 
 	satisfying $(g_n-h_n)(0)=0$ and $(g_n-h_n)'(0)=0$. It follows that $(g_n-h_n)'(t)> 0$ for every $t> 0$, 
 	and therefore the fraction on the right-hand side of \eqref{eq:n-hmgDetT} is strictly 
 	less than 1 provided that $\lambda({\bf y})>0$. That this is indeed the case follows from part (g) of 
 	Lemma \ref{lem:n-convexity}, since $h_n(1)>0$. Thus $\ab{\det {\bf T}'({\bf y})}<\la({\bf y})^{d(n-1)-2}$. 
 	The exponent $d(n-1)-2$ is nonnegative as long as $d\geq 2$ and $n\geq 2$, or $d=1$ and $n\geq 
 	3$.
 	For such pairs $(d,n)$, the proof is finished by noting that $\la({\bf y})\leq 
 	1$.
	 	
 	To handle the remaining case $(d,n)=(1,2)$, start by noting that
 	$$\det {\bf S}'({\bf T}(y))\det {\bf T}'(y)=\det {\bf T}'({\bf S}(y))\det {\bf S}'(y)=1,\text{ for every }y\in\R\setminus\{0\},$$
	since ${\bf T}$ and ${\bf S}$ are inverse maps. 
 	Therefore, it suffices to show that  $\det {\bf S}'(y)>1$,  for 
 	every $y\neq 0$. From \eqref{quadratic-T'S'} and \eqref{quadratic-rho},
  	we obtain 
 	\[ \det {\bf S}'(y)=\Bigl(\frac{g_n(1)}{h_n(1)}\Bigr)^{-\frac{1}{2}}\frac{g_n'(1)}{h_n'(1)}, \]
 	so we are left with checking that
 	\[ \frac{g_n(1)}{h_n(1)}<\biggl(\frac{g_n'(1)}{h_n'(1)}\biggr)^2.\]
 	In the present case, $h_n'(1)=2h_n(1)$, and so this can be rewritten as
 	\[ 4h_n(1)g_n(1)< (g_n'(1))^2. \]
 	Since $\phi$ is strictly convex, we can 
 	write $g_n(t)=h_n(t)+H_n(t)$, where 
 	$H_n$ is strictly convex and $H_n(0)=0$. 
	Thus it suffices to check that
 	\[ 4h_n(1)(h_n(1)+H_n(1))< (2h_n(1)+H_n'(1))^2, \]
 	or equivalently
 	\[ 4h_n(1)H_n(1)< 4h_n(1)H_n'(1)+(H_n'(1))^2. \]
 	This last inequality holds if $H_n(1)\leq H_n'(1)$. 
	But this is immediate since 
 	$H_n(0)=0$ and $H_n$ is convex. In particular,  $H_n'(1)\geq H_n'(t)$, for every 
 	$t\in[0,1]$, and so
 	\[ H_n(1)=H_n(0)+\int_0^1 H_n'(t)\d t\leq H_n'(1). \]
	This completes the proof.
 \end{proof}
 
 \noindent With the right tools at our disposal, the proof of Theorem \ref{thm:Thm7} now follows similar lines to that of \cite[Theorem 1.3]{OSQ16}.
 We provide the details for the convenience of the reader.

\begin{proof}[Proof of  Theorem \ref{thm:Thm7}]
	As in the proof of Proposition \ref{prop:ConvolutionProperties}, we may write
	\begin{align}
	\nu^{\ast(n)}(\boldsymbol{\xi},\tau)&=\int_{\R^{d(n-1)}}\ddirac{\tau-n\psi(\boldsymbol{\xi}/n)-g_n(1,{\bf y})}\d 
	{\bf y},\label{n-PsiConvolution}\\
	\nu_0^{\ast(n)}(\boldsymbol{\xi},\tau)&=\int_{\R^{d(n-1)}}  
	\ddirac{\tau-n\vphi(\boldsymbol{\xi}/n)-h_n(1,{\bf y})}\d {\bf y}.\notag
	\end{align}
	By \eqref{eq:SupportConvolution}, the support of  
	$\nu^{\ast(n)}$ is contained in the support of 
	$\nu_0^{\ast(n)}$.
	 For each $\boldsymbol{\xi}\in\R^d$,  consider the map ${\bf T}$ given by 
	\eqref{eq:n-T}, which by  Lemma \ref{lem:transformation-general} maps each ellipsoid 
	$\mathcal E_\vphi (\boldsymbol{\xi},c)$
	bijectively onto $\mathcal E_\psi (\boldsymbol{\xi},c)$,  for every $c>0$.
	Changing variables ${\bf y}\rightsquigarrow {\bf T}({\bf y})$ in \eqref{n-PsiConvolution}, and appealing to the defining identity \eqref{eq:DefLambdaViaHG}, 
	\begin{align}
	\nu^{\ast(n)}(\boldsymbol{\xi}, \tau)
	&=\int_{\R^{d(n-1)}}\ddirac{\tau-n\psi(\boldsymbol{\xi}/n)-g_n(1,{\bf T}({\bf y}))}|\det {\bf T}'({\bf y})|\d {\bf y}\notag\\
	&=\int_{\R^{d(n-1)}}\ddirac{\tau-n\psi(\boldsymbol{\xi}/n)-g_n(\la({\bf y}),{\bf y})}|\det {\bf T}'({\bf y})|\d {\bf y}\notag\\
	&=\int_{\R^{d(n-1)}}\ddirac{\tau-n\psi(\boldsymbol{\xi}/n)-h_n(1,{\bf y})}|\det {\bf T}'({\bf y})|\d {\bf y}\notag\\
	&=\int_{\R^{d(n-1)}}\ddirac{(\tau-n\phi(\boldsymbol{\xi}/n))-n\vphi(\boldsymbol{\xi}/n)-h_n(1,{\bf y})}|\det {\bf T}'({\bf y})|\d
	{\bf y}.\label{n-beforestrict}
	\end{align}
	From Lemma \ref{lem:n-contractive}, we know that  $|\det {\bf T}'|\leq 1$, and so H\"older's inequality 
	implies
	\[\nu^{\ast(n)}(\boldsymbol{\xi},\tau)\leq
	\nu_0^{\ast(n)}(\boldsymbol{\xi},\tau-n\phi(\boldsymbol{\xi}/n)),\]
	for every $\boldsymbol{\xi}\in\R^d$ and $\tau>n\psi(\boldsymbol{\xi}/n)$. Thus inequality \eqref{eq:GeometricComparisonMainIneq} holds.
	We now appeal to \eqref{n-<1} to argue that this inequality must be strict at every 
	point in the interior of the support of $\nu^{\ast(n)}$. Let  $(\boldsymbol{\xi},\tau)$ be one such 
	point, for which $c:=\tau-n\psi(\boldsymbol{\xi}/n)>0$. 
	The singular measure that is being integrated in \eqref{n-beforestrict} is supported on the ellipsoid 
	$\mathcal E_\vphi(\boldsymbol{\xi},c)$. Since $c>0$, this ellipsoid does not contain the origin, and by 
	Lemma \ref{lem:n-contractive}  the strict inequality $|\det {\bf T}'({\bf y})|<1$ holds at every point ${\bf y}\in 
	\mathcal E_\vphi(\boldsymbol{\xi},c)$. This can be strengthened to $|\det {\bf T}'({\bf y})|\leq c_0$ for some fixed 
	$c_0<1$ (which depends on $\phi,\boldsymbol{\xi},\tau$ but not on ${\bf y}$), since the set $\mathcal 
	E_\vphi(\boldsymbol{\xi},c)$ is compact and the function ${\bf y}\mapsto\det {\bf T}'({\bf y})$ is continuous. The result now 
	follows from replacing the $\delta$-function appearing in the integral \eqref{n-beforestrict} 
	by 
	an appropriate $\varepsilon$-neighborhood of the ellipsoid $\mathcal E_\vphi(\boldsymbol{\xi},c)$, and then 
	analyzing 	the cases of equality in H\"older's inequality. To conclude the proof of the theorem, let 
	$\varepsilon\to 0^+$.
\end{proof}
 
\section{Convex perturbations of parabolas}\label{sec:Perturbations}
In this section, we deduce Theorem \ref{thm:Thm6} from Theorem \ref{thm:Thm7}.
Let $\phi:\R\to\R$ satisfy the conditions of Theorem \ref{thm:Thm6}, let $\nu$ denote projection measure on the curve $\Sigma_\phi\subset\R^2$ defined in \eqref{eq:DefSigmaPhi}, and set $\psi:=|\cdot|^2+\phi$.
The following result is a direct analogue of \cite[Lemmata 4.1 and 4.2]{OSQ16}, and can be proved in the same way.

\begin{lemma}
 \label{lem:upper-bound-concentration}
 Given $y_0\in\R$,  let $\{f_n\}\subset L^2(\R)$ be a sequence concentrating at $y_0$. 
 Then 
  \begin{equation}
  \label{value_boundary}
  \limsup_{n\to\infty}\frac{\norma{f_n\nu\ast f_n\nu\ast f_n\nu}_{L^2(\R^{2})}^2}{\norma{f_n}_{L^2(\R)}^6}\leq 
  (\nu\ast\nu\ast\nu)(3y_0,3\psi(y_0)).
 \end{equation}
 If we set  $f_n(y)=e^{-n(\psi(y)-\psi(y_0)-\psi'(y_0)(y-y_0))}$, then  the sequence 
$\{f_n\norma{f_n}^{-1}_{L^2}\}\subset L^2(\R)$ concentrates at $y_0$, and
equality holds in \eqref{value_boundary}.
\end{lemma}

\begin{proof}[Proof of Theorem \ref{thm:Thm6}]
Denote the optimal constant in inequality \eqref{eq:PertParabolaConvolutionForm} by
$${\bf P}_\phi:=\sup_{0\neq f\in L^2}\frac{\|f\nu\ast f\nu\ast f\nu\|^{1/3}_{L^2(\R^2)}}{ \|f\|_{L^2(\R)}}.$$
We first show that ${\bf P}^6_\phi={\pi}/{\sqrt{3}}$.
The Cauchy--Schwarz inequality implies 
\begin{equation}\label{eq:firstCS}
\norma{f\nu\ast f\nu\ast f\nu}^2_{L^2(\R^{2})}\leq 
\norma{\nu\ast\nu\ast\nu}_{L^\infty(\R^2)}\norma{f}_{L^2(\R)}^6.
\end{equation}
On the other hand, the convolution  $\nu\ast\nu\ast\nu$ defines a bounded function on $\R^2$, as can be seen from  identity
 \eqref{eq:ConvolutionFormula}: one just applies the integral version of the 
Mean Value Theorem, after noting that $\psi''=2+\phi''$ and $\phi''\geq 0$. 
As a consequence,
${\bf P}_\phi^6\leq \norma{\nu\ast\nu\ast\nu}_{L^\infty}.$
Now, let $\nu_0$ denote the projection measure on the parabola $\Sigma_0$.
 From Theorem \ref{thm:Thm7} and Remark \ref{rem:tripleconvolution}, we know that 
 $$\norma{\nu\ast\nu\ast\nu}_{L^\infty(\R^2)}\leq\norma{\nu_0\ast\nu_0\ast\nu_0}_{L^\infty(\R^2)}=\tfrac{\pi}{\sqrt{3}}.$$ 
 Therefore ${\bf P}_\phi^6\leq {\pi}/{\sqrt{3}}$.
In order to show that
 ${\bf P}_\phi^6\geq {\pi}/{\sqrt{3}}$,
 we use the sequence given by \eqref{eq:ExtSeqPertParabola}. 
 In case (i), since $\phi''(y_0)=0$, it follows from \eqref{eq:Formula3foldConvolution} 
 that 
 $$(\nu\ast\nu\ast\nu)(3y_0,3\psi(y_0))=\tfrac{\pi}{\sqrt{3}}.$$
 Thus the sequence $\{f_n\|f_n\|^{-1}_{L^2}\}$, where
 $$f_n(y)=e^{-n(\psi(y)-\psi(y_0)-\psi'(y_0)(y-y_0))},$$
 is extremizing for \eqref{eq:PertParabolaConvolutionForm} in light of Lemma \ref{lem:upper-bound-concentration}. 
 In case (ii), we have  
$$(\nu\ast\nu\ast\nu)(3y_n,3\psi(y_n))\to\tfrac{\pi}{\sqrt{3}}, \text{ as } n\to\infty.$$
Choose a sequence $\{a_n\}\subset\N$ in such a way that the function 
$$f_n(y)=e^{-a_n(\psi(y)-\psi(y_n)-\psi'(y_n)(y-y_n))}$$
satisfies
 \[\left\vert{\frac{\norma{f_{n}\nu\ast f_n\nu\ast 
f_{n}\nu}_{L^2(\R^2)}^2}{\norma{f_{n}}_{L^2(\R)}^6}-(\nu\ast\nu\ast\nu)(3y_n,3\psi(y_n))}
\right\vert\leq\frac{1}{n},\]
\begin{equation}\label{conctrinfty}
\int_{|y-y_n|\geq \frac1n} |f_{n}(y)|^2 \d y \leq \frac{1}{n} \|f_{n}\|_{L^2(\R)}^2,
\end{equation}
for every $n\in\N$.
That this is possible follows again from Lemma \ref{lem:upper-bound-concentration}. Since 
$$\frac{\norma{f_{n}\nu\ast f_n\nu\ast
f_{n}\nu}_{L^2(\R^2)}^2}{\norma{f_{n}}_{L^2(\R)}^6}\to\tfrac{\pi}{\sqrt{3}}, \textrm{ as } n\to\infty,$$
the sequence $\{f_n\|f_n\|^{-1}_{L^2}\}$ is again extremizing for  \eqref{eq:PertParabolaConvolutionForm}.  
It follows that ${\bf P}_\phi^6= {\pi}/{\sqrt{3}}$.
 
 We finish by showing that extremizers for \eqref{eq:PertParabolaConvolutionForm} do not exist. Aiming at a contradiction, let $f\geq 0$ be an extremizer. By an application of Cauchy--Schwarz and H\"older's inequalities, 
\begin{align*}
{{\bf P}_\phi^6} \norma{f}_{L^2(\R)}^6
&=\norma{f\nu\ast f\nu\ast f\nu}_{L^2(\R^2)}^2\\
&\leq \int_{\R^2} |(f^2\nu\ast f^2\nu\ast f^2\nu)(\xi,\tau)|(\nu\ast\nu\ast \nu)(\xi,\tau)\d\xi \d\tau\\
&\leq \|\nu\ast\nu\ast\nu\|_{L^\infty(\R^2)}\int_{\R^2} |(f^2\nu\ast f^2\nu\ast f^2\nu)(\xi,\tau)|\d\xi \d\tau\\
&= \|\nu\ast\nu\ast\nu\|_{L^\infty(\R^2)} \|f\|_{L^2(\R)}^6.
 \end{align*}
 Since ${\bf P}_\phi^6=\|\nu\ast\nu\ast\nu\|_{L^\infty}={\pi}/{\sqrt{3}}$ and 
 $f\neq 0$, all inequalities in this chain of inequalities must be equalities. In 
particular, the convolution $\nu\ast\nu\ast\nu$ must be constant equal to 
$\|\nu\ast\nu\ast\nu\|_{L^\infty}$ almost everywhere inside the support of 
 $f^2\nu\ast f^2\nu\ast f^2\nu$, which is a set of positive Lebesgue measure since $f\neq 
 0$. This contradicts the strict inequality 
 \begin{equation*}
(\nu\ast\nu\ast\nu)(\xi, \tau)<\norma{\nu\ast\nu\ast\nu}_{L^\infty(\R^2)},\textrm{ for 
almost every }(\xi, \tau)\in\textrm{supp} (\nu\ast\nu\ast\nu),
\end{equation*}
 which in turn follows from the second part of Theorem \ref{thm:Thm7}. This contradiction shows that extremizers do not exist. The proof of the theorem is now complete.
\end{proof}

\end{document}